\documentclass[12pt]{amsart}
\usepackage{amscd}
\usepackage{amssymb,amsmath,amsthm}
\usepackage[all]{xy}

\theoremstyle{plain}
\newtheorem{thm}{Theorem}[section]
\newtheorem{lem}[thm]{Lemma}

\newtheorem{cor}[thm]{Corollary}
\newtheorem{defn-lem}[thm]{Definition-Lemma}

\newtheorem{prop}[thm]{Proposition}

\theoremstyle{definition}
\newtheorem{defn}[thm]{Definition}
\newtheorem{ex}[thm]{Example}

\def\md #1#2#3#4#5 {\left(
                        \begin{matrix}
             #1 & #2 \\
             #3 & #4
                        \end{matrix}
                      \right)- #5}

\def\ma #1#2#3#4 {\left(
                        \begin{matrix}
             #1 & #2 \\
             #3 & #4
                        \end{matrix}
                      \right)}
\def\dist{\operatorname{dist}}
\def\Index{\operatorname{Index}}
\def\tsr{\operatorname{tsr}}

\def\id{\operatorname{id}}

\def\Cu{\operatorname{Cu}}
\def\TA{\operatorname{TA}}

\def\CAR{\operatorname{CAR}}
\newcommand{\mc}{\mathcal}
\begin{document}
\title [A tracially sequentially-split $\sp*$-homomorphisms between $C\sp*$-algebras]
       { Tracially sequentially-split $\sp*$-homomorphisms between $C\sp*$-algebras}

\begin{abstract}
We define a tracial analog of the sequentially split $*$-homomorphism between $C\sp*$-algebras of Barlak and Szab\'{o} and show that several important approximation properties related to the classification theory of $C\sp*$-algebras pass from the target algebra to the domain algebra.  Then we show that this framework arises from tracial Rokhlin type properties of $C\sp*$-algebras with respect to a finite group action and an inclusion of $C\sp*$-algebras.
\end{abstract}

\author { Hyun Ho \, Lee and Hiroyuki Osaka}

\address {Department of Mathematics\\
          University of Ulsan\\
         Ulsan, South Korea 44610 \\
 and \\
School of Mathematics\\ Korea Institute of Advanced Study\\
Seoul, South Korea, 130-772 }
\email{hadamard@ulsan.ac.kr}

\address{Department of Mathematical Sciences\\
Ritsumeikan University\\
Kusatsu, Shiga 525-8577, Japan}
\email{osaka@se.ritsumei.ac.jp}


\date{}
\thanks{The first author's research was supported by Basic Science Research Program through the National Research Foundation of Korea(NRF) funded by the Ministry of Education(NRF-2018R1D1A1B07050924). \\
The second author's research was partially supported by the JSPS grant for Scientific Research No. 17K05285.}
\maketitle

\section{Introduction}

 Related to Elliott's classification program, the local approximation by a class of $C\sp*$-algebras has replaced the inductive limit construction. More generally, a tracial version of local approximation, which roughly means a uniform decomposition of finite elements into a building block in the designated class and a tracially small part, has been initiated by H. Lin \cite{Lin:tracial} (We refer the reader to \cite{EN} for the formal definition).   Inspired by  a successful classification of tracially AF $C\sp*$-algebras \cite{Lin:classification} a tracial version of Rokhlin property for a finite group action on a unital $C\sp*$-algebra  was suggested by N. C. Phillips \cite{Phillips:tracial} based on Izumi's strict Rokhlin property of a finite group action \cite{Izumi:finite} and the second named author and T. Teruya \cite{OT1} formulated the tracial Rokhlin property  for an inclusion of unital $C\sp*$-algebras $P\subset A$ of index-finite type; the former is concerned with the crossed product by an action which is  larger than the original algebra and the latter is concerned with a subalgebra which is smaller than the original algebra. Nonetheless, it would be desirable to have a framework to deal with both actions and inclusions  together.  For the strict Rokhlin property, such a framework has already appeared in the name of sequentially split $*$-homomorphism between $C\sp*$-algebras \cite{BS}. Following the idea of  \cite{Lin:tracial, Phillips:tracial, OT1} we suggest a tracial version of the sequentially split $*$-homomorphism between unital $C\sp*$-algebras using the the central sequence algebra and Cuntz subequivalence relation.

One of most desirable regularity properties in classification of nuclear simple $C\sp*$-algebras is $\mc{Z}$-absorption and it has been known that a crossed product of $\mc{Z}$-absorbing $C\sp*$-algebra by a finite group action  with the tracial Rokhlin property  is again $\mc{Z}$-absorbing and the target algebra of a conditional expectation with a tracial Rokhlin property is $\mc{Z}$-absorbing if the domain algebra is $\mc{Z}$-absorbing.  Encompassing these results we show that if there is a tracially sequentially split $*$-homomorphism from a simple unital $C\sp*$-algebra $A$  to a simple unital $C\sp*$-algebra $B$ and $B$ is $\mc{Z}$-absorbing then $A$ is $\mc{Z}$-absorbing too. Moreover we  investigate whether other approximation properties pass from $B$ to $A$. Our main results are as follows; suppose that there is a tracially sequentially split $*$-homomorphism from a unital separable  simple $C\sp*$-algebra $A$ to a unital separable simple $C\sp*$-algebra $B$. If $B$ satisfies the following properties, then $A$ does too. 
\begin{itemize}
\item finite or stable rank one
\item real rank zero
\item strict comparison 
\item tracially approximated by a class of semi-projective $C\sp*$-algebras, in particular  tracial topological rank zero.  
\item tracially  $\mc{Z}$-absorbing
\end{itemize}

We would like to point out some remarks related to our notion. Though our notion about the tracially sequentially-split $*$-homomorphism is obtained by weakening the rigidity of an approximate left inverse, the arguments and methods are substantially different with ones used in \cite{BS}. In addition, to show that approximation properties pass from the target algebra to the domain algebra we assume that the target algebra is simple in most cases. We do not know whether our results still hold even if we drop the simplicity of target algebras though we could provide a counterexample for the stable rank one case (see Example \ref{E:tsr} below).  We also note that our notion is not well suited to other local approximation properties such as nuclear dimension less than $n$, which are stable under a functorial operation $\infty: (A_{\infty}) \to (A_{\infty})_{\infty}$ in \cite{BS}, thus more algebraic in nature.  
   
\section{Tracially sequentially-split homomorphism between $C\sp*$-algebras and permanence properties}
For $A$  a $C\sp*$-algebra, we let $A^{+}$ be the cone of positive elements including $0$ and set 
\[c_0(\mathbb{N}, A)=\{(a_n)_n\mid \lim_{n\to \infty}\|a_n \|=0 \} \]
\[l^{\infty}(\mathbb{N}, A)=\{ (a_n)_n \mid  \text{a sequence}\, (\|a_n\|)_n \,\text{bounded}  \}\]
Then we denote by $A_{\infty}=l^{\infty}(\mathbb{N}, A)/c_0(\mathbb{N}, A)$ the sequence algebra of $A$ with the norm  of $a \in A_{\infty}$ given by $\limsup_n \|a_n\| $, where $(a_n)_n$ is a representing sequence of $a$ or $a=[(a_n)_n]$. We can embed $A$ into $A_{\infty}$ as a constant sequence, and we denote the central sequence algebra of $A$ by 
\[A_{\infty} \cap A'.\]
For an automorphism of $\alpha$, we also denote by $\alpha_{\infty}$ the induced automorphism on $A_{\infty}$. 

In this paper we are interested in a class $\mc{C}$ closed under some conditions that could be varied by different purposes through the following formulation due to H. Lin. 

 \begin{defn} \cite{EN}\label{D:TAS}
A simple unital $C\sp*$-algebra $A$ belongs to the class $\TA \mc{C}$ if for every finite set $\mc{F}\subset A$ and every $\epsilon>0$, and every nonzero $a\in A^{+}$, there exist a projection $p\in A$ and a $C\sp*$-subalgebra $C (\subset A)$ in $\mc{C}$ with $1_C=p$ such that for all $x\in \mc{F}$
\begin{enumerate}
\item$ \| px -xp\| \le \epsilon$,
\item $\dist(pxp, C)\le \epsilon$,
\item $1-p$ is Murray-von Neumann equivalent to a projection in $\overline{aAa}$.
\end{enumerate}
If $A$ satisfies the first two conditions, then we say $A$ is locally approximated by $C\sp*$-algebras unitarily in $\mc{C}$. When we consider the class $\TA \mc{C}$ which satisfies the strict comparison property (see the paragraph after Lemma \ref{L:sub} in this article or \cite{ET}), the following alternative condition is suggested to check the third condition from time to time,  
\begin{itemize}
\item [(3)'] $\tau(1-p)\le \epsilon$ for all $\tau\in T(A)$  the convex set of tracial states of $A$.
\end{itemize}

\end{defn}\begin{defn} \label{D:C}
Let $\mc{C}$ be a class of unital separable (weakly) semi-projective $C\sp*$-algebras whose quotients can be locally approximated by $C\sp*$-algebras unitarily in $\mc{C}$ such that 
\begin{enumerate}
\item if $A\in \mc{C}$, then $M_n(A)\in \mc{C}$,
\item if $A\in \mc{C}$, then any hereditary subalgebra of $A$ is in $\mc{C}$.
\end{enumerate}
\end{defn}
On the other hand, the second named author and Teruya \cite{OT1} consider the following class $\mc{S}$.  
\begin{defn}\label{D:S}
Let $\mc{S}$ be a class of unital separable $C\sp*$-algebras. Then $\mc{S}$ is \emph{finitely saturated} if the following closure conditions hold: 
\begin{enumerate}
\item If $A \in \mc{S}$ and $B\cong A$, then $B\in \mc{S}$.
\item If $A_1,A_2, \dots, A_n \in \mc{S}$, then $\bigoplus_{i=1}^nA_i \in \mc{S}$.
\item If $A \in \mc{S}$, then $M_n(A) \in \mc{S}$.
\item If $A\in \mc{S}$ and $p\in A$ a nonzero projection, then $pAp\in \mc{S}$.  
\end{enumerate} 
Moreover,  the finite saturation of $\mc{S}$ is the smallest finite saturated class which contains $\mc{S}$.  
\end{defn}

Elliott and Niu \cite{EN} considered a class for $\mc{S}$, which consists of all interval algebras and closed under  a finite direct sum.  In Elliott classification program, the most natural examples of $\mc{C}$ are the set of finite dimensional $C\sp*$-algebras and the set of so called 1-dimensional NCCW, which are of the form $F_1 \oplus (F_2\otimes C[0,1])$ where $F_i$'s are finite dimensional $C\sp*$-algebras. 
In the following definition, the class $\mc{C}$ is not necessarily same as one  in Definition \ref{D:C} or Definition \ref{D:S}. 

Though it is well known to experts,  we would like to observe when the conditions (1), (2), (3)  and the conditions (1), (2), (3)'  in Definition \ref{D:TAS} are equivalent since the literature is often scattered (see for instance \cite{ELPW, Lin:TR}). It is necessary that a $C\sp*$-algebra in our consideration has enough projections, thus we restrict ourselves on the following class of $C\sp*$-algebras. 
\begin{defn}\label{D:SP}
A $C\sp*$-algebra $A$ has  Property (SP) or (SP)-property if any nonzero hereditary $C\sp*$-subalgebra of $A$ has a nonzero projection. 
\end{defn}
  
\begin{prop}
Suppose that a simple unital $C\sp*$-algebra $A$ satisfies (SP)-property and the strict comparison property. Then the conditions (1), (2), (3) and the conditions (1), (2), (3)' for $A$ to be $\TA \mc{C}$ are equivalent. 
\end{prop}
\begin{proof}
Assume that $A$ satisfies (1), (2), (3) in Definition \ref{D:TAS}. Then given a finite set $\mc{F}$, $\epsilon>0$, we take $n\in \mathbb{N}$ such that $1/n < \epsilon$. Then we choose $n$-mutually orthogonal projections $r_1, \dots, r_n$ in $A$ such that $r_i$ is Murray-von Neumann equivalent to $r_j$ for any $i$ and $j$. Note that $\tau(r_1)< \epsilon$ for any tracial state $\tau \in T(A)$. 
Then take $a=r_1$, then we have a projection $p$ and a $C\sp*$-algebra $C$ from the class $\mc{C}$ satisfying (1), (2). By (3), we have $1-p$  is Murray-von Neumann equivalent to a projection smaller than $r_i$. Thus $\tau(1-p)\le \tau(r_1) < \epsilon$ so (3)' is obtained. \\
Conversely, since $A$ is unital, $T(A)$ is compact. Then given a finite set $\mc{F}$, $\epsilon>0$, $a\in A_{+}\setminus \{0\}$ we take a nonzero projection $r$ in $\overline{aAa}$ and consider $\delta=\min\{\epsilon, \inf_{\tau \in T(A)}\{\tau(r)\}\}$. Then for $\mc{F}$ and $\delta$  we have a projection $p$ and a $C\sp*$-algebra $C$ from the class $\mc{C}$ satisfying (1) and (2). By (3)' we have 
\[ \tau(1-p) < \delta < \tau(r)\] for any $\tau \in T(A)$. 
Then the strict comparison property implies that $1-p $ is Cuntz subequivalent to $r$, and $r$ is in $\overline{aAa}$, so we are done.  
\end{proof}

Now we introduce the key concept of this paper.  We define $A_{\infty}^{++}:=A_{\infty}^{+} \setminus \{[(x_n)] \in A_{\infty} \mid  \exists \{k_n\}_{n \in \mathbb{N}} \subset \mathbb{N} \,\, \text{such that } \lim_{n\to \infty} x_{k_n}=0\}$. Then, any nonzero projection $p=[(p_n)] \in A^{++}_{\infty}$ means that each $p_n$ is a nonzero projection in $A$. Similarly, any nonzero positive element $a=[(a_n)] \in A^{++}_{\infty}$ means that each $a_n$ is a nonzero positive in $A$ and the norm of $\|a_n\|$ is uniformly away from zero. Suppose that this statement is false. Then for each $1/n$  we can find an element $a_{k_n}$ such that $\|a_{k_n} \| < \frac{1}{n}$ and thus  $\lim_{n\to \infty}a_{k_n}=0$.  Then $a \notin A^{++}_{\infty}$.  
\begin{defn}\cite[Definition 2.1]{LeeOsaka} \label{D:Traciallyseqsplit}
Let $A$ and $B$ be separable $C\sp*$-algebras and $A$ be unital. A $*$-homomorphism $\phi:A \to B$ is called tracially sequentially-split, if for any nonzero element $z \in A_{\infty}^{++}$  there exist a $*$-homomorphism $\psi: B \to A_{\infty}$ and a nonzero projection $g\in A_{\infty}\cap A'$ such that 
\begin{enumerate}
\item $\psi(\phi(a))=ag$ for each $a\in A$, 
\item $1_{A_{\infty}} -g$ is Murray-von Neumann equivalent to a projection in a hereditary $C\sp*$-subalgebra $\overline{zA_{\infty}z}$ in $A_{\infty}$.
\end{enumerate}
\end{defn}

Note that if $B$ is unital and $\phi$ is unit preserving, then $g=\psi(1_B)$.  In the case of  $g=1_{A_{\infty}}$, $\phi$ is called a (strictly) sequentially split $*$-homomorphism where the second condition is automatic \cite[Page 11]{BS}. The $\psi$ in the above definition is called a tracial approximate left-inverse of $\phi$. Although the diagram below is not commutative, we still use it to symbolize the definition of a tracially sequentially-split map $\phi$ and its tracial approximate left inverse $\psi$;

\begin{equation}\label{D:diagram}
\xymatrix{ A \ar[rd]_{\phi} \ar@{-->}[rr]^{\iota} && A_{\infty} &\\
                          & B \ar[ur]_{\psi} }
\end{equation}
The following observation will be used throughout the paper. 

\begin{lem}\label{L:full}
\begin{enumerate}
\item Let $x$ be a full  element in $A_{\infty}$ and  write $x=[(x_n)_n]$. Then $x\in A_{\infty}^{++}$.
\item If $A$ is simple and $x \in A_{\infty}\cap A'$ is a nonzero projection, $ax$ is full for any $a\in A\setminus \{0\}$. In particular, when $A$ is unital, any nonzero projection $x$ in $A_{\infty}\cap A'$ is full and any nonzero $a\in A^{+}$ is full in $A_{\infty}$. 
\end{enumerate}
\end{lem}
\begin{proof}
(1) Suppose that there exists a strictly increasing subsequence $(i_n)_{n}$ of natural numbers such that $\lim_{n \to \infty}x_{i(n)}=0$. Then we can construct a map $\zeta:l^{\infty}(\mathbb{N}, A) \to l^{\infty}(\mathbb{N}, A)$ defined by $\zeta((y_n)_n)=((y_{i(n)})_n)$. Since this map induces a map from $A_{\infty}$ to $A_{\infty}$, we still denote the map by $\zeta$. It is obvious that $\zeta$ is nontrivial and $\zeta(x)=0$. Since $x$ is full, the closed ideal generated by $x$ is $A_{\infty}$. It follows that $\zeta$ becomes trivial, which is a contradiction.\\
(2)  Suppose that $ax$ is not full. Then the ideal generated by $ax$, say $I$,  is proper. Consider a quotient map $\pi: A_{\infty} \to A_{\infty}/I$ and combine it with the map $\rho: A \to A_{\infty}$ which sends $b\in A$ to $bx \in A_{\infty}$. Since $A$ is simple, the $*$-homomorphism $\pi \circ \rho: A \to A_{\infty}/I$ is injective.  Hence $(\pi \circ \rho) (a)=0$ implies $a=0$ which is a contradiction.  
\end{proof}

\begin{defn}
We say that $A_{\infty}$ has a weak (SP)-property if for any nonzero  element $x\in A^{++}_{\infty}$ the hereditary $C\sp*$-subalgebra generated by $x$ contains a nonzero projection. 
\end{defn}

\begin{prop}\label{P:dichotomy}
Let $A$ and $B$  be separable unital $C\sp*$-algebras and $\phi:A\to B$ be a tracially sequentially-split $*$-homomorphism. Then $A_{\infty}$ has a weak (SP)-property or $\phi$ is a (strictly) sequentially split $*$-homomorphism.
\end{prop}
\begin{proof}
Suppose that $A_{\infty}$ does not have a weak (SP)-property. Then there exists a nonzero $x\in A^{++}_{\infty}$ such that the hereditary subalgebra $\overline{xA_{\infty}x}$ contains no nonzero projections.   Since $\phi$ is tracially sequentially-split, we have $\psi: B \to A_{\infty}$ and a projection $g$ in $A_{\infty}\cap A'$ such that $\psi(\phi(a))=ag$ and $1-g $ is Murray von-Neumann equivalent to a projection in  $\overline{xA_{\infty}x} $. Since $\overline{xA_{\infty}x}$ has no nonzero projections, it follows that $g=1_{A_{\infty}}$. 
\end{proof}

We recall that a unital $C\sp*$-algebra has stable rank one if the invertible elements are dense in $A$ and denoted by $\tsr(A)=1$ shortly. 
 \begin{lem}\label{L:SP}
Let $A$ be a separable unital simple $C\sp*$-algebra.  Suppose that $A_{\infty}$ has a weak (SP)-propery. Then $A$ has (SP)-property. The converse is also true. 
\end{lem}
\begin{proof}
Let $x\in A^{+}\setminus \{0\}$. Then $x$ is full in $A_{\infty}$ by Lemma \ref{L:full}. Thus $x \in A^{++}_{\infty}$.  Consider $\overline{xA_{\infty}x}$. By the  assumption it contains a nonzero projection which implies that $\overline{xAx}$ contains a nonzero projection. Since $A$ is separable, every hereditary $C\sp*$-subalgebra of $A$ is of the form $\overline{xAx}$ and we are done.\\
 Conversely, assume that $A$ has (SP)-property. For a nonzero $x \in A^{++}_{\infty}$, we can represents it $x=[(x_n)_n]$ where $x_n\neq 0$ for all $n$. Then  for each $n$ we can find a nonzero projection $p_n \lesssim x_n$ so that there exists $y_n (\neq 0)$ such that $\| y_n x_n y_n^*-p_n \| < \frac{1}{n}$. Now let $y=[(y_n)_n]$ and $p=[(p_n)_n]$. Then $yxy^*=p$.  Thus $p \lesssim x$. 
\end{proof} 
\begin{thm}\label{T:TSR1}
Let both $A$ and $B$ be separable simple unital $C\sp*$-algebras and assume that there exists a unital tracially sequentially-split $*$-homomorphism $\phi:A\to B$. If $B$ is finite or more strongly $\tsr(B)=1$, then  $A$ is finite or $\tsr(A)=1$ respectively. 
\end{thm}
\begin{proof}
We assume that $A_{\infty}$ satisfies a weak (SP)-property, otherwise the conclusion follows from \cite{BS}.  Assume that $B$ is finite. Then we first show that $A$ is finite. Suppose that $vv^*=1$ for $v\in A$. Then $\phi(v)\phi(v)^*=1$ in $B$. Thus $\phi(v^*v)=1$. Then for $1_{A_{\infty}}$ we have a tracial approximate left inverse $\psi$ and obtain $(v^*v-1)g=0$ for a nonzero projection $g \in A_{\infty}\cap A'$. Since $A$ is simple, it follows that $v^*v=1$.   Hence $A$ is finite. Next we assume that $\tsr(B)=1$. Then to show that $\tsr(A)=1$, in view of \cite[Proposition 3.2]{Ro:UHF1}, it is enough to show that we can approximate a two sided zero divisor $x$ by an invertible element in $A$.  

 Let $y$ be an element in $A$ such that $yx=xy=0$ and $\epsilon >0$ given. We consider a hereditary $C\sp*$-algebra generated by $y$ which contains a nonzero projection $e$. Then we can find orthogonal subprojections $e_1$ and $e_2$ such that $e_1+e_2=e$ and $e_1 \sim e_2$ since $A$ is simple and has (SP)-property.  Then $yx=0$ implies that $xe=0$, thus $xe_1=0$. Similarly, $e_1x=0$. So $x=(1-e_1)x(1-e_1)$ and $\phi(x) \in \overline{(1-\phi(e_1))B (1-\phi(e_1))}$.  Since  the stable rank of the hereditary $C\sp*$-subalgebra of $B$ is preserved,  we can find an invertible element $b$ in $\overline{(1-\phi(e_1))B (1-\phi(e_1))}$ such that \[\| \phi(x)-b \| < \epsilon /3. \]
Then we take a tracial approximate inverse $\psi:B \to A_{\infty}$ such that $1-\psi(1)$ is Murray-von Neumann equivalent to a projection in $\overline{e_2A_{\infty}e_2}$, i.e., $1-\psi(1) \lesssim e_2$.  Then 
\[ \| \psi(b)-\psi(\phi(x))   \| =\| \psi(b) -xg  \|< \epsilon /3 \] where $g=\psi(1) \in A_{\infty}\cap A'$. 
In addition, there is a partial isometry $w$ in $A_{\infty}$ such that $w^*w=1-g$ and $ww^* \le e_1$. 
Let $v=w(1-e_1)$. Note that $v^*v=(1-g)(1-e_1)$ and $vv^* \le e_1$. \\
Consider $c=(1-g)x(1-g)+\epsilon/6(v+v^*+e_1-vv^*)$, in a matrix form  with respect to $(1-g)(1-e_1)$, $g(1-e_1)$, $e_1$; 
\[ \left(
\begin{matrix}
(1-g)x(1-g) & 0 & \frac{\epsilon}{6} v^* & 0 \\
0 &0 &0 & 0 \\
\frac{\epsilon}{6}v&0&0&0\\
0&0&0& \frac{\epsilon}{6}(e_1-vv^*)
\end{matrix}
\right).
\] 
It is not difficult to verify that $c$ is invertible element in $(1-g)(1-e_1)A_{\infty}(1-e_1)(1-g)+e_1A_{\infty}e_1$. Then $\psi(b)+c$ is invertible element in $A_{\infty}$ and 
\begin{align*}
\|x-(\psi(b)+c)  \| &= \|xg +(1-g)x(1-g) - (\psi(b)+c) \| \\
& \le \| xg -\psi(b)\|+ \epsilon /6 \|v+v^*+e_1-vv^*\|\\
& \le \epsilon/3+\epsilon/2 < \epsilon. 
\end{align*}
\end{proof}

If $B$ is non-simple, then Theorem \ref{T:TSR1} is not true in general since we find the following example. 
\begin{ex}\label{E:tsr}
Let $\alpha$ be an action of $\CAR$ with $\alpha^2 = \id$ constructed by 
Blackadar.(Symmetry of the $\CAR$ algebra, Annals of Math. 131 (1990), 
589--623). Note that $\alpha$ has the tracial Rokhlin property by \cite[Example 
3.1]{Phillips:tracial}. Then Proposition \ref{P:actiontoinclusion} and  Theorem \ref{T:tracialinclusion} imply that $\iota: {\CAR^\alpha} \rightarrow \CAR$ is a tracially sequentially-split $*$-homomorphism. It follows that $\id \otimes \iota:C[0,1] \otimes 
\CAR^{\alpha} \rightarrow C[0,1] \otimes \CAR$ is also a tracially sequentially-split 
$*$-homomorphism. It is obvious that $\tsr(C[0,1] \otimes \CAR) =1$. On the other hand, since 
$K_1(\CAR^{\alpha})$ is not trivial, $\tsr(C[0,1] \otimes \CAR^{\alpha}) \neq 1$ by 
\cite[Proposition 5.2]{NOP:ranks}.
\end{ex}

 We also recall that a unital $C\sp*$-algebra has real rank zero if the invertible self-adjoint elements are dense in the space of self-adjoint elements of $A$. 

\begin{thm}\label{T:RR0}
Let $A$ and $B$ separable simple unital $C\sp*$-algebras and assume that there exists  a unital tracially sequentially-split $*$-homomorphism $\phi:A\to B$. If $B$ has real rank zero, then so does $A$. 
\end{thm}
\begin{proof}
In view of Proposition \ref{P:dichotomy} and Lemma \ref{L:SP} we may assume that $A$ has  (SP)-property, otherwise the conclusion follows from \cite[Theorem 2.9]{BS}. Let $x$ be a self-adjoint element in $A$, then we consider a continuous function $f$ as defined below; 
\[ f(t)= \begin{cases}
1 \quad\text{if $|t| \le \frac{\epsilon}{12}$}, \\
0 \quad \text{if $|t| \ge \frac{\epsilon}{6}$},\\ 
\text{linear when $\frac{\epsilon}{12} \le |t| \le \frac{\epsilon}{6}$}.
\end{cases}\]
We may assume that $f(x) \ne 0$. Then  $\overline{f(x)Af(x)}$ the hereditary $C\sp*$-algebra  contains a nonzero projection $e$ and there are mutually orthogonal projections $e_1$ and $e_2$ such that $e=e_1+e_2$ and $e_2 \lesssim e_1$. Note that  $\|yx \| \le \frac{\epsilon}{6}$ for any $y \in \overline{f(x)Af(x)}$. It follows that $\displaystyle \| x-(1-e_1)x(1-e_1)\| \le \frac{\epsilon}{6}$. 
Again we have a tracial approximate inverse $\psi:B \to A_{\infty}$ such that $\psi(\phi(a))=ag$ and $1-g$ is Murray-von Neumann equivalent to a projection in the hereditary subalgebra generated by $e_2$ in $A_{\infty}$ where $g=\psi(1) \in A_{\infty}\cap A'$. 
Then we can find a self-adjoint  invertible element $b$ in $\overline{(1-\phi(e_1))B(1-\phi(e_1))}$ such that  
\[  \| (1-\phi(e_1))\phi(x)(1-\phi(e_1))-b  \| <  \frac{\epsilon}{6}, \] and take $c$ as in the proof of Theorem \ref{T:TSR1}.  Note that $c$ is a self-adjoint element. Then $\psi(b)+c$ is a self-adjoint invertible element in $A_{\infty}$ and 
  \begin{align*}
\|x-(\psi(b)+c)  \| &= \|x-(1-e_1)x(1-e_1)+(1-e_1)x(1-e_1) - \psi(b)-c \| \\
& \le \frac{\epsilon}{6} +\| (1-e_1)x(1-e_1)g -\psi(b)\|+ \frac{\epsilon}{6} \|v+v^*+e_1-vv^*\|\\
& \le \epsilon/3+\epsilon/2 < \epsilon. 
\end{align*}
\end{proof}

\begin{thm}\label{T:TAC}
Let $A$ and  $B$ be separable simple unital $C\sp*$-algebras and assume that  there exists a tracially sequentially-split $*$- homomorphism  $\phi:A \to B$. Let $\mc{C}$ be a class as in Definition \ref{D:C} which satisfies the strict comparisin property. If $B$ belongs to $\TA \mc{C}$, then so does $A$.
\end{thm}
\begin{proof}
It is enough to show that given a finite set $F\subset A$, $\epsilon>0$, there is a unital $C\sp*$-subalgebra $C$ of $A$ in $\mc{C}$ and a projection $p=1_C$ such that for all $x\in F$
\begin{enumerate}
\item $\| px -xp\|] \le \epsilon$,
\item $\dist(pxp, C)\le \epsilon$, 
\item $\tau(p)\ge 1-\epsilon$ for all $\tau \in T(A)$.
\end{enumerate} 

Consider a finite set $\phi(F) \subset B$. Then since $B$ belongs to $\TA\mc{C}$, we have a unital $C\sp*$-algebra $B_1 \subset B$ and $B_1\in \mc{C}$ such that for all $x\in F$ 
\begin{enumerate}
\item[(a)] $\|1_{B_{1}}\phi(x)-\phi(x)1_{B_1}\| \le \epsilon$, 
\item[(b)] $\dist(1_{B_1}\phi(x)1_{B_1}, B_1)\le \epsilon$,
\item[(c)] $\tau(1_{B_1})\ge 1-\epsilon$ for all $\tau \in T(B)$. 
\end{enumerate}

Let $\psi$ be the tracial approximate left inverse for $\phi$ and  consider its restriction to $B_1$ $\psi|_{B_1}:B_1\to A_{\infty}$ . Since $B_1$ is a semi-projective $C\sp*$-algebra, there exist  a lift $\kappa: B_1 \to l^{\infty}(\mathbb{N}, A)$ such that the following diagram commutes. 
\[\xymatrix{ & l^{\infty}(\mathbb{N}, A) \ar[d]\\
B_1 \ar[ur]^{\kappa} \ar[r]_{\psi} & A_{\infty}}
\]
We write $\kappa=(\kappa_1,\kappa_2,\dots )$ where each $\kappa_n$ is a $*$-homomorphism from $B_1$ to $A$. Note that 
\[\begin{split}
[\psi(1_{B_1}), \psi(\phi(x))]&=\psi(1_{B_1})x\psi(1_B)-x\psi(1_B)\psi(1_{B_1})\\
&=\psi(1_{B_1})\psi(1_B)x -x\psi(1_B)\psi(1_{B_1})\\
&=\psi(1_{B_1})x -x\psi(1_{B_1})\\
&=[\psi(1_{B_1}),x]
\end{split} 
\]
Thus we have for all $x\in F$
\begin{equation*}\label{E:TAC(1)}
\|[\psi(1_{B_1}), x] \| \le \epsilon
\end{equation*}
It follows that for all $x\in F$ 
\begin{equation}
\limsup_{n} \| [\kappa_n (1_{B_1}), x] \| \le \epsilon
\end{equation}
Moreover, $\dist(\psi(1_{B_1})\psi(\phi(x))\psi(1_{B_1}), \psi(B_1))=\dist(\psi(1_{B_1})x\psi(1_{B_1}), \psi(B_1)) \le \epsilon$. 
It follows that 
\begin{equation}
\limsup_{n}\dist(\kappa_n(1_{B_1})x\kappa_n(1_{B_1}), \kappa_n(B_1)) \le \epsilon
\end{equation}
Also, for any $\tau \in T(A)$, we have a  trace $\tau_{\omega}$ on $A_{\infty}$ defined by 
\[\tau_{\omega}([(a_n)_n])=\lim_{n\to \omega} \tau(a_n)\] for a free ultrafilter $\omega$ on $\mathbb{N}$. Then
$\tau'=\tau_{\omega} \circ \psi \in T(B)$ and  $\tau'\circ \phi=\tau$. 

By the condition (c) we have 
\begin{equation}\label{E:traceestimate}
\tau_{\omega}(\psi(1_{B_1}))\ge 1-\epsilon.
\end{equation}

Thus  we conclude that 
\begin{equation}\label{E:pretracialestimate}
\liminf_n \min_{\tau \in T(A)}\tau(\kappa_n(1_{B_1}))\ge 1-\epsilon.
\end{equation}
Therefore, we have the following estimates for a suitable $n$, 
\begin{enumerate}
\item[(1)'] $\| [\kappa_n(1_{B_1}), x] \| \le 2\epsilon $ for all $x\in F$,
\item[(2)'] $\dist(\kappa_n(1_{B_1})x\kappa_n(1_{B_1}), \kappa_n(B_1))\le 2\epsilon$ for al $x\in F$,
\item[(3)'] $\tau(\kappa_n(1_{B_1}))\ge 1-2\epsilon$.
\end{enumerate}
 Consider a finite set $G\subset \kappa_n(B_1)$ consisting of $y$'s such that
$\displaystyle \|\kappa_n(1_{B_1})x\kappa_n(1_{B_1})-y\| \le 2\epsilon$ for each $x\in F$ by (2)'.
Since $\mc{C}$ is a class such that any quotient of $C\sp*$-algebra in $\mc{C}$ is locally approximated by a $C\sp*$-algebra in $\mc{C}$ unitarily, we have a $C\sp*$-algebra $C \in \mc{C}$ such that $C\subset \kappa_n(B_1)$ with $1_C=\kappa_n(1_{B_1})$ such that for any $y\in G$
\begin{equation}\label{E:localapproximation}
\dist(y, C) \le \epsilon
\end{equation} 
 Consequently, we show that there exist a $C\sp*$-subalgebra $C \subset A$ such that for all $x\in F$
\begin{enumerate}
\item $\|[1_C, x] \| \le 2\epsilon$, 
\item $\dist(1_Cx1_C, C)\le 3\epsilon$, 
\item $\tau(1_C)\ge 1-2\epsilon$.
\end{enumerate}
So we are done. 
\end{proof}

Because Toms-Winter conjecture is yet confirmed (the implication from the strict comparison property to $\mc{Z}$-stability is not complete in full generality) \cite{SWW:Toms-Winter}, we provide the following observation based on \cite{OT1} which is applicable even in the absence of traces.

\begin{thm}\label{T:TAS}
Let $A$ and $B$ be simple separable unital $C\sp*$-algebras and $\mc{S}$ be a class of unital (weakly) semi-projective $C\sp*$-algebras satisfying the conditions as in Definition \ref{D:S} (in fact, we do not need the condition (2)). Let $\phi:A \to B$ be a tracially sequentially-split $*$-homomorphism. If $B\in \TA\mc{S}$, then so does $A$. 
\end{thm}
\begin{proof}
Again we assume that $A$ has (SP)-property or $A_{\infty}$ has a weak (SP)-property. Consider a triple $(F, \epsilon, a)$ where $F$ is a finite set, $\epsilon>0$, and $a$ is a nonzero positive element in $A$.  We must show that there exist a projection $p$ and $C\sp*$-subalgebra  $C (\subset A)$ in $\mc{S}$ with $1_C=p$ such that for all $x\in F$ 
\begin{enumerate}
\item $\|px-xp \| \le \epsilon$,  
\item $ \dist(pxp, C) \le \epsilon$,  
\item $1-p$ is Murray-von Neumann equivalent to a projection in $\overline{aAa}$. 
\end{enumerate} 
We can take two orthogonal nonzero projections $r_1$ and  $r_2$ in $\overline{aAa}$. Then $\phi(r_1)\perp \phi(r_2)$.  For a triple $(G=\phi(F), \epsilon, \phi(r_1))$, since $B\in \TA\mc{S}$, we have a projection $p'$ and a $C\sp*$-algebra $C'$ in $\mc{S}$ with $1_{C'}=p'$ such that for all $x\in F$ 
\begin{enumerate}
\item $\|p'\phi(x)-\phi(x)p'\| \le \epsilon$,
\item $\dist (p'\phi(x)p', C') \le \epsilon$,
\item $1-p'$ is Murray-von Neumann equivalent to a projection $s$ in $\overline{\phi(r_1)B\phi(r_1)}$.
\end{enumerate} 
Then for $r_2>0$, there are a projection $g\in A_{\infty}\cap A'$ and a $*$-homomorphism $\psi:B\to A_{\infty}$ such that 
\begin{itemize}
\item[(i)] $\psi(\phi(a))=ag$, 
\item[(ii)] $1-g \sim r $ where $r$ is a projection in $\overline{r_2A_{\infty}r_2}$ in $A_{\infty}$.
\end{itemize}
Since $C'$ is (weakly) semi-projective, we can lift a restriction of $\psi$ to $C'$ denoted by $\psi |_{C'}$ to $l^{\infty}(\mathbb{N}, A)$.  In other words, we have a sequence of $*$-homomorphism $\psi_k$ from $C'$ to $A$ such that $\psi(c)=(\psi_k(c))_k + C_0(\mathbb{N}, A)$. Note that each $\psi_k$  is injective  since $B$ is simple and $\psi|_{C'}$ is injective.  So $\psi_k(C') \in \mc{S}$.\\
Then from (i) for all $x\in F$
\[\| \psi(p')x-x\psi(p')\| \le \frac{\epsilon}{2}. \]
 Or, 
\begin{equation}
\limsup_k \|\psi_k(p')x-x\psi_k(p') \|  \le \frac{\epsilon}{2}. 
\end{equation}
From (ii) for all $x\in F$
\[ \dist(\psi(p')x\psi(p'), \psi(C')) \le \frac{\epsilon}{2} \]
Or, 
\begin{equation}
 \limsup_k [\dist(\psi_k(p')x\psi_k(p'), \psi_k(C'))] \le \frac{\epsilon}{2}. 
\end{equation}
Moreover,
\[ 
\begin{split}
1-\psi(p')&=1-g +g-\psi(p')\\
&= 1 - g + \psi(1 - p')\\
&\sim r+\psi(s) \in \overline{aA_{\infty}a}.
\end{split}
\]
Then for large enough $k$ we can conclude  for all $x$
\begin{enumerate}
\item  $\|\psi_k(1_{C'})x-\psi(1_{C'}) \| \le \epsilon$, 
\item $\dist(\psi_k(1_{C'})x\psi_k(1_{C'}), \psi_k(C')) \le \epsilon $
\end{enumerate}
together with the fact that 
$1-\psi_k(1_{C'}) $ is Murray-von Neumann equivalent to a projection $\overline{aAa}$ in $A$.  
\end{proof}

Now we turn to the strict comparison property and $\mc{Z}$-stability.  
\begin{lem}\label{L:sub}
Let $A$ be a unital $C\sp*$-algebra and $B$ a hereditary $C\sp*$-subalgebra. Suppose that $r,q$ are two projections in $A$ such that $r \le q $ and $q \sim p$ in $A$ where $p$ is a projection in $ B$. Then there is a projection $r' \in B$ such that $r\sim r'$ and $r' \le p$.  
\end{lem}
\begin{proof}
Let $v$ be a partial isometry in $A$ such that $vv^*=q$, $v^*v=p$. Then take $r'=v^*rv$ and it follows that $r'$ is a projection and $r \sim r'$. Note that 
$pr'=r'=r'p$,  so $r'\le p$. The fact $p \in B$ implies that $r' \in B$.  
\end{proof}

 For two positive elements $x,y$ in a unital $C\sp*$-algebra $A$ we say that $x$ is Cuntz subequivalent to $y$, written as $x\lesssim y$, if there exist sequences $(r_n)_n$ and $(s_n)_n$ such that $r_nys_n \to x$ in norm \cite{Cu}.  

For each $\epsilon >0 $ define $f_{\epsilon}:\mathbb{R}_{+} \to \mathbb{R}_{+} $ by 
\[
f_{\epsilon}(t)=
\begin{cases}
0, \quad &0\le t \le \epsilon,\\
\epsilon^{-1}(t-\epsilon) \quad &\epsilon \le t  \le 2\epsilon, \\
1 \quad & 2\epsilon \le t. 
\end{cases}
\]
The following characterization will be useful \cite[Proposition 2.4]{Ro:UHF2}. 
\begin{prop}[Rordam]\label{P:Subequivalence}
The following are equivalent; 
\begin{enumerate}
\item $x \lesssim y$. 
\item There exists a sequnce $r_n$ in $A$ such that $r_nyr^*_n \to x$.
\item   For every $\epsilon >0$ there exist $\delta>0 , r\in A$ such that $r f_{\delta} (y)r^* =f_{\epsilon}(x) $.
\end{enumerate}
\end{prop}
\begin{lem}\cite[Proposition 1.1]{Cu} \label{L:useful}
Let $x,x',y, y' \in A^{+}$ be such that $x \lesssim y, x' \lesssim y'$ and $y \perp y'$. Then $x+x' \lesssim y+y'$.  
\end{lem}

Next we recall the definition of the strict comparison property; Let $A$ be a unital nuclear $C\sp*$-algebra, and denote by  $T(A)$ the space of normalized traces on $A$. Given $\tau \in T(A)$, we define a lower semicontinuous map $d_{\tau}:M_{\infty}(A)_{+} \to \mathbb{R}^{+}$ by 
\[d_{\tau}(a)=\lim_{n \to \infty} \tau(a^{1/n}).\]  where $M_{\infty}(A)_{+}$ denotes the positive elements in $M_{\infty}(A)$ that is the algebraic limit of the directed system $(M_n(A), \phi_n)$.
 If $A$ has the property that $a\lesssim b$ whenever $d_{\tau}(a) < d_{\tau}(b)$ for every $\tau \in T(A)$, then we say $A$ has \emph{the strict comparison property of positive elements} or shortly \emph{strict comparison}.
\begin{defn}[Hirshberg and Orovitz]
We say that a unital $C\sp*$-algebra $A$ is tracially $\mc{Z}$-absorbing if $A \ncong \mathbb{C}$ and for any finite set $F\subset A$, $\epsilon >0$, and nonzero positive element $a\in A$ and $n\in \mathbb{N}$ there is an order zero contraction $\phi:M_n(\mathbb{C}) \to A$ such that the following hold:
\begin{enumerate}
\item $1-\phi(1)\lesssim a$, 
\item for any normalized element $x\in M_n(\mathbb{C})$ and any $y\in F$ we have $\|[\phi(x), y]\| < \epsilon $.
\end{enumerate}
\end{defn}

Recall that the Jiang-Su algebra $\mc{Z}$ is a simple separable nuclear and infinite dimensional $C\sp*$-algebra with a unique trace and the same Elliott invariant with $\mathbb{C}$ \cite{JS}. We say that $A$ is $\mc{Z}$-stable or $\mc{Z}$-absorbing if $A\otimes \mc{Z}\cong A$. 
\begin{thm}\label{T:Z-absorbing}
Let $A$ and $B$ be two simple unital $C\sp*$-algebras and $\phi:A\to B$  a unital tracially sequentially-split  $*$-homomorphism. If $B$ is tracially $\mc{Z}$-absorbing, then so is $A$. Thus, if $B$ is $\mc{Z}$-absorbing, then $A$ is also $\mc{Z}$-absorbing provided that $A$ is nuclear.    
\end{thm}
\begin{proof}
Let $F$ be a finite set of $A$, $\epsilon >0$, $n \in \mathbb{N}$, $z$ be a nonzero positive element in $A$. We may assume $A$ has Property (SP) or $A_{\infty}$ has a weak (SP)-property. (Otherwise $\phi$ is sequentially split, so  in the below argument $\psi$ is the approximate left inverse of $\phi$ such that $\psi(1_B)=g=1_{A_{\infty}}$ and for $\phi(z) \neq 0$ we consider an order zero map $\phi^{'}: M_n \to B$ directly.)  So there are mutually orthogonal nonzero projections $p_1, p_2$ in $\overline{zAz}$ a hereditary subalgebra of $A$ generated by $z$.\\
Set $G=\phi(F)$ a finite set in $B$, then for $\phi(p_1)$ there is an order zero contraction  $\phi':M_n(\mathbb{C}) \to B$ such that 
\begin{enumerate}
\item $1-\phi'(1) \lesssim \phi(p_1)$,
\item $\forall x \in M_n(\mathbb{C})$ such that $\|x\|=1$, $\|[\phi'(x), y]\| < \epsilon$ for every $y\in G$.
\end{enumerate}  
Consider $\widetilde{\psi}=\psi \circ \phi': M_n(\mathbb{C}) \to A_{\infty}$ which is an order zero contraction where $\psi$ is a tracial approximate inverse with a projection $g\in A_{\infty}\cap A'$ such that  $1-g \lesssim p_2$ in $A_{\infty}$.
Then 
\[
\begin{split}
1-\widetilde{\psi}(1)&=1-g +g-\psi(\phi'(1))\\
&=1-g+\psi(1-\phi'(1))\\
&\lesssim 1-g +p_1g \\
&\lesssim p_2+p_1 \lesssim z.
\end{split}
\]
Moreover, for $a\in F$
\[
\begin{split}
[\widetilde{\psi}(x), a]&=\psi(\phi'(x))a-a\psi(\phi'(x))\\
&=\psi(\phi'(x))ga-ag\psi(\phi'(x)) \quad \text{where $\psi(1)=g$}\\
&=\psi(\phi'(x))\psi(\phi(a))-\psi(\phi(a))\psi(\phi'(x))\\
&=\psi([\phi'(x), \phi(a)]).
\end{split}
\]
Therefore 
\begin{equation} \label{E:estimatefororderzeromap}
\|[\widetilde{\psi}(x), a]\| < \epsilon. 
\end{equation}
Since $C^*(\phi'(M_n(\mathbb{C}))) \subset B$ is semi-projective, there is a lift $\widehat{\psi}: C^*(\phi'(M_n(\mathbb{C}))) \to l^{\infty}(\mathbb{N}, A)$ of $\psi$. We write 
\[ \widehat{\psi}(b)=(\widehat{\psi}_n(b))+C_0(\mathbb{N}, A)\] where $\widehat{\psi}_n$ is a $*$-homomorphism from $C^*(\phi'(M_n(\mathbb{C})))$ to $A$ and write $g=[(g_n)_n]$ where $\widehat{\psi}_n(1)=g_n$. \\
Since $\limsup_n \|\widehat{\psi}_n(\phi(1))-p_1g_n \|=0$, we can take a large enough $n$ so that $\|\widehat{\psi}_n(\phi(1))-p_1g_n \|$ is very small. Then there is a projection $p_1' \in \overline{p_1Ap_1}$ satisfying $\widehat{\psi}_n(\phi(1))\sim p_1'$. 
Simultaneously, we also control  $\| [\widehat{\psi}_n(\phi'(x)), a]\| < 2\epsilon$ from (\ref{E:estimatefororderzeromap}). 
Then 
\[
\begin{split}
1-\widehat{\psi}_n(\phi'(1))&=1-g_n+g_n -\widehat{\psi}_n(\phi'(1))\\
&=1-g_n+\widehat{\psi}_n(1-\phi'(1))\\
&\lesssim 1-g_n +\widehat{\psi}_n(\phi(p_1))\\
&\lesssim p_2 +p'_1\\
&\lesssim z.
\end{split}
\]
The last statement follows from \cite[Proposition 2.2]{HO} and \cite[Theorem 4.1]{HO}.
\end{proof}

\begin{prop}\label{P:cuntz}
Let $A, B$ be separable simple unital $C\sp*$-algebras and assume that there is $\phi:A \to B$ a tracially sequentially-split $*$-homomorphism. If $\phi(a) \lesssim \phi(b)$ for two positive elements $a, b \in A$ with the condition that $0$ is not isolated point of $\sigma(b)$ the spectrum of $b$, then $a\lesssim b$.
\end{prop}
\begin{proof}
Let $g_{\epsilon}(t)=\max\{t-\epsilon, 0\}$. Then the positive part of $a-\epsilon 1_A$ denoted by $(a-\epsilon)_{+}$ is equal to $g_{\epsilon}(a)$. We want to show that $(a-\epsilon)_{+} \lesssim b$ for every $\epsilon >0$. 
Since $\phi$ is a unital map,  note that \[(\phi(a)-\epsilon1_B)_{+}=g_{\epsilon}(\phi(a))=\phi(g_{\epsilon}(a)).\]
Since $\phi(a) \lesssim \phi(b)$ in $B$, there exists $\delta >0$ and  $r\in B$ such that $(\phi(a)-\epsilon)_{+}=r^*(\phi(b)-\delta)_{+}r \sim (\phi(b)-\delta)^{1/2}_{+}r ((\phi(b)-\delta)^{1/2}_{+}r)^*=b_0$ where the latter belongs to $\overline{\phi(b)B\phi(b)}$. 
Since $\sigma(b)\cap (0,\delta) \neq \emptyset$,  we can take a nonzero positive element $c\in A$ such that $(b-\delta)_{+}\perp c$ and take a tracially approximate inverse $\psi$ such that $1-\psi(1) \lesssim c$ in $A_{\infty}$. Then $\psi(b_0)\in \overline{bA_{\infty}b}$ and $\psi(b_0)\perp c$.\\
Then  with $g=\psi(1)$
\[
\begin{split}
(a-\epsilon)_{+}&=(a-\epsilon)_{+}g +(a-\epsilon)_{+}(1-g)\\
&=\psi(\phi((a-\epsilon)_{+}))+(a-\epsilon)_+(1-g)\\
&\lesssim \psi(b_0)+c \\
&\lesssim b \quad \text{in $A_{\infty}$.}
\end{split}
\]
Consequently, $(a-\epsilon)_+\lesssim b$ in $A$, so we are done. 
\end{proof}

We can obtain $a\lesssim b$ without the extra condition that $0$ is not isolated in $\sigma(b)$ in Proposition \ref{P:cuntz} if $d_{\tau}(a) < d_{\tau}(b)$.  But we need some preparations. We first recall that $M_{\infty}(A)$ is the algebraic direct limit of the system $(M_n(A), \phi_n)$ where $\phi_n:M_n(A) \to M_{n+1}(A)$ are usual embeddings given by $a \mapsto a\oplus 0$ (see Definition \ref{D:Cuntz} below) . Then we can define $a \lesssim b$ as in Proposition \ref{P:Subequivalence} for $a ,b \in M_{\infty}(A)^{+}$ or $a,b \in (A\otimes \mathbb{K})^{+}$. We remark that under the usual convention that $A \subset M_n(A) \subset M_{\infty}(A) \subset A\otimes \mathbb{K}$  for $a,b \in A$ we have $a \lesssim b$ in $A$ if $a \lesssim b$ in $M_{\infty}(A)$ or in $A\otimes \mathbb{K}$.  
\begin{defn}\label{D:Cuntz}
Let $A$ be a $C\sp*$-algebra. With a binary operation $a \oplus b= \left( \begin{matrix} a &0 \\ 0 & b \end{matrix}\right)$ for $a\in M_n(A)$ and $b\in M_{m}(A)$ we define the commutative semigroup $W(A)= M_{\infty}(A)/ \sim $ together with the semigroup operation $\langle a \rangle + \langle b \rangle = \langle a \oplus b \rangle$ and the partial order $\langle a \rangle \le \langle b \rangle \leftrightarrow a \lesssim b$. Similarly $\Cu(A)= (A\otimes \mathbb{K})^{+}/ \sim$.  
\end{defn}
\begin{defn}
Let $\Cu(A)_{+}$ denote the set of elements $\eta \in \Cu(A)$ which are not the classes of projections. Similarly, let $W_{+}(A)$ denote the set of elements $\eta \in W(A)$ which are not the classes of projections. Further call an element $a \in (A\otimes \mathbb{K})^{+} ( \in M_{\infty}(A)^{+})$ \emph{purely positive} if $\langle a \rangle \in \Cu(A)_{+}( \in W(A)_{+})$ respectively.
\end{defn}
\begin{lem} \cite[Lemma 3.2]{Phillips:large} \label{L:purelypositive}
Let $A$ be a stably finite simple unital $C\sp*$-algebra. $a \in (A\otimes \mathbb{K})^{+}$ is purely positive if and only if $0$ is not isolated point of $\sigma(a)$. 
\end{lem} 
\begin{lem}\cite[Lemma 3.6]{Phillips:large}\label{L:alternative}
Let $A$ be a stably finite simple unital $C\sp*$-algebra which is not of type I.  Let $p \in A\otimes \mathbb{K}$ be a nonzero projection, let $n$ be a positive integer and $\xi \in \Cu(A)\setminus \{0\}$. Then there exist $\mu, \kappa \in W_{+}(A)$ such that $\mu \le \langle p \rangle \le \mu+ \kappa  $ and $n\kappa \le \xi$. 
\end{lem}
 
\begin{thm}\label{T:Strictcomparision}
Let $A, B$ be stably finite, simple, nuclear(exact), separable, infinite dimensional  $C\sp*$-algebras and $\phi:A \to B$ be a tracially sequentially-split $*$-homomorphism.  If $B$ satisfies the strict comparison property, so does $A$.   
\end{thm}
\begin{proof}
Since $\phi$ is tracially sequentially-split,  for each $n$  there exist  a projection $g_n \in A_{\infty}\cap A'$ and a $*$-homomorphism $\psi_n:B \to A_{\infty}$ such that $\tau(1-g_n)< 1/n$ for every $\tau \in T(A_{\infty})$ and $\psi_n(\phi(a))=ag_n$ for all $a\in A$.   Then we claim that $T(\phi): T(B) \to T(A)$ is surjective. Let $\tau$ be a trace in $A$. Then we consider the weak-$*$ limit of traces $\tau_{\infty} \circ \psi_n$ in $T(B)$ denoted by $w^*-\lim (\tau_{\infty}\circ \psi_n)$ where $\tau_{\infty}$ is the induced trace on $A_{\infty}$. Then  for $a\in A$
\[
\begin{split}
T(\phi)([w^*-\lim (\tau_{\infty}\circ \psi_n)])(a)&=[w^*-\lim (\tau_{\infty} \circ \psi_n)](\phi(a)))\\
&=\lim_n \tau_{\infty}(\psi_n(\phi(a)))\\
&=\lim \tau_{\infty}(ag_n)\\
&=\tau_{\infty}(a)-\lim_n \tau_{\infty}(a(1-g_n))\\
&=\tau(a).
\end{split} 
\]
Now given two positive elements $a,b$ in $A$ assume that  $d_\tau(a) < d_{\tau}(b)$ for all  $\tau\in T(A)$. So $d_{\tau}(\phi(a)) < d_{\tau}(\phi(b))$ for all $\tau \in T(B)$. Since $B$ satisfies the strict comparison property, it follows that $\phi(a) \lesssim \phi(b)$. \\ 
 Then we split two cases; when $b$ is purely positive, then by Lemma \ref{L:purelypositive} and Proposition \ref{P:cuntz}, $a\lesssim b$. If $b$ is not purely positive, we may assume that $b$ is Cuntz equivalent to a projection $p$. Now consider again two cases;
\begin{itemize}  
\item[(i)] $a$ is Cuntz equivalent to a projection,
\item[(ii)] $a$ is not  Cuntz equivalent to a projection.
\end{itemize}

In the case of (i) the same proof works. That is, $\langle b\rangle -\langle a \rangle$ is continuous on $T(A)$.

In the case of (ii)  $a$ is purely positive. Therefore, for any $\epsilon > 0$ take a function $f:[0, \infty) \rightarrow [0,1]$ such that 

$f(\lambda) > 0 $ for $\lambda \in (0, \epsilon)$ and $f(\lambda) = 0$ for $\lambda \in \{0\} \cup [\epsilon, \infty)$. Then, $f(a) \not= 0$. Therefore, $\rho = \inf\{d_\tau (f(a)) \colon \tau \in \mathrm{T}(A)\}$ satisfies $\rho > 0$.

We claim that for any $\tau \in \mathrm{T}(A)$, $d_\tau((a - \epsilon)_+) + \rho < d_\tau(b)$.

Indeed,
\begin{align*}
d_\tau((a-\epsilon)_+) + \rho
&\leq d_\tau((a-\epsilon)_+) + d_\tau(f(a))\\
&= d_\tau((a-\epsilon)_+ + f(a)) \leq d_\tau(a) < d_\tau(b).
\end{align*}
Choose $n \in \mathbb{N}$ such that $\frac{1}{n} < \rho$. Use Lemma \ref{L:alternative} to find $c, d \in (bAb \otimes \mathbb{K})_+ \backslash\{0\}$ such that $\langle c \rangle \leq \langle b \rangle \leq \langle c\rangle + \langle d\rangle$, and $n\langle d\rangle \leq \langle 1\rangle$ in $\Cu(A)$, where $c$ is purely positive. Note that for $\tau \in \mathrm{T}(A)$, $d_\tau(\langle d\rangle) \leq \frac{1}{n} < \rho$, $d_\tau(\langle c\rangle) > d_\tau(\langle b\rangle) - \rho> d_\tau(\langle (a - \epsilon)_+\rangle)$. Again,  by Lemma \ref{L:purelypositive} and Proposition \ref{P:cuntz} $(a - \epsilon)_+ \preceq c$ in $A$ since $c$ is purely positive.

Therefore, from $\langle c\rangle \leq \langle b\rangle$ in $\Cu(A)$, $(a - \epsilon)_+ \preceq b$ in $A$. Since $\epsilon > 0$ is arbitrary, we conclude that $a \preceq b$ in $A$.
\end{proof}

\section{Applications}
\begin{defn}\cite[Definition 2.5]{LeeOsaka}
Let $A$ and $B$ be unital $C\sp*$-algebras and $G$ a discrete group. Given two actions  $\alpha:G \curvearrowright A$, $\beta:G \curvearrowright B$ respectively, an equivariant $*$-homomorphism $\phi:(A, \alpha) \to (B, \beta) $ is called  $G$-tracially sequentially-split, if for every nonzero positive element $z\in A_{\infty}$ there exists an equivariant tracial approximate left inverse $\psi:(B, \beta) \to (A_{\infty},\alpha_{\infty})$. 
\end{defn}
The following definition is due to N.C. Philiips. 
\begin{defn}\cite[Definition 1.2]{Phillips:tracial})
Let $G$ be a finite group and $A$  an infinite dimensional simple separable unital finite $C\sp*$-algebra. We say that $\alpha: G \curvearrowright A$ has the tracial Rokhlin property if  for every finite set $F \subset A$, every $\epsilon>0$, any nonzero positive element $x\in A$ there exist $\{e_g\}_{g\in G}$ mutually orthogonal projections such that 
\begin{enumerate}
\item $\| \alpha_g(e_h)-e_{gh} \| \le \epsilon$, \quad $\forall g, h \in g$,
\item $\|  e_ga -ae_g \| \le \epsilon$, \quad $\forall g \in G$, $\forall a \in F$,
\item  Write $e=\sum_{g} e_g$, and $1-e$ is Murray-von Neumann equivalent to a projection in $\overline{xAx}$. 
\end{enumerate}
\end{defn}

Then the following reformulation is well-known and appeared in \cite{LeeOsaka} without a proof. Here we include a proof. 

\begin{thm}\cite[Theorem 4.2]{LeeOsaka}\label{T:tracialRokhlinaction}
Let $G$ be a finite group and $A$ an infinite dimensional separable  simple unital  finite $C\sp*$-algebra. Then $\alpha: G \curvearrowright A$ has the tracial Rokhlin property if  and only if for any nonzero positive element $x \in A^{++}_\infty$ there exist a mutually orthogonal  projections  $e_g$'s in $A_{\infty}\cap A'$ such that 
\begin{enumerate}
\item $\alpha_{\infty, g}(e_h)=e_{gh}$, \quad $\forall g,h\in G$ where $\alpha_{\infty}:G \curvearrowright A_{\infty} $ is the induced action,  
\item 
$1-\sum_{g} e_g$ is Murray-von Neumann equivalent to a projection in $\overline{xA_{\infty}x}$.
\end{enumerate}
\end{thm}
\begin{proof}
Let  $x (\neq 0)$ in $A^{++}_{\infty}$. Then we can represent it as a sequence $(x_n)_n \in l^{\infty}(A)$ where for each $n$ $x_n \neq 0$ and $\| x_n \|>\delta >0$ for some $\delta$ since we are allowed to do so up to  $C_0(\mathbb{N}, A)$. Since $A$ is separable, we can also take an increasing sequence of finite sets $F_1 \subset F_2 \subset \cdots $ so that $\displaystyle \overline{\cup_n F_n } =A$. Then for each $n$  there exist  mutually orthogonal projections $\{e_{g,n}\}_{g\in G}$ such that  
\begin{enumerate}
\item $\| \alpha_g(e_{h,n})-e_{gh,n} \| \le \frac{1}{n}$, \quad $\forall g, h \in g$,
\item $\|  e_{g,n}a -ae_{g,n} \| \le \frac{1}{n}$, \quad $\forall g \in G$, $\forall a \in F_n$,
\item $1-\sum_{g}e_{g,n}$ is Murray-von Neumann equivalent to a projection $q_n$ in $\overline{x_nAx_n}$. 
\end{enumerate} 
Now we take $e_g=[(e_{g,n})_n]$. Then it is easily shown that $e_g$'s are mutually orthogonal projections in $A_{\infty}\cap A'$ and $\alpha_{\infty,g}(e_h)=e_{gh}$. Also
\[ (1-\sum_g e_{g,n})_n \sim (q_n)_n \in \overline{(x_n)_n l^{\infty}(\mathbb{N}, A)(x_n)_n}. \] Then it follows that $1-\sum_g e_g \sim [(q_n)_n] \in \overline {xA_{\infty}x}$. 

For the other direction,  for a finite set $F$, $\epsilon>0$, and a positive nonzero $x$ in $A$, we embed $x$ in $A_{\infty}$ as a constant sequence and  denote it by $x$ again.  Then there are mutually orthogonal projections $e_g=[(e_{g,n})_n]$ in a central sequence algebra such that  $1-\sum_g e_g$ is Murray-von Neumann equivalent to a projection $r =[(r_n)_n]$ in $\overline{xA_{\infty}x}=(\overline{xAx})_{\infty}$. Then 
we have 
\[ \begin{aligned}
&(1)\, \limsup_{n \to \infty} \| e_{g,n}a-ae_{g,n} \|=0 \quad \text{for every $a \in A$}\\
&(2)\, \limsup_{n \to \infty} \| \alpha_{h}(e_{g,n})- e_{hg,n}\|=0  \quad \text{for every $g, h\in G$}\\
&(3)\, \limsup_{n \to \infty} \|   e_{g,n}e_{g,n}-e_{g,n}\|=0   \quad \text{for every $g\in G$}\\
&(4)\, \limsup_{n \to \infty} \| e_{g,n}e_{h,n}\|=0  \quad \text{for every $g \neq h \in G$}
\end{aligned}
\]
Also, there is a sequence $(v_n)_n$ such that \[(5)\,\limsup \|(1-\sum_g e_{g,n})- v_n^*v_n\|=\limsup \| r_n -v_nv_n^*    \|=0. \] Using (3) and (4) for a large enough $n$ we can assume $e_{g,n}$'s are mutually orthogonal projections within the difference $\epsilon/4|G|$ where $|G|$ is the cardinality of the finite group $G$. In other words, there exist $N$ such that if $n\ge N$ then  for $g\in G$ we can take $e_{g,n}'$ a projection such that $\|e_{g,n}'-e_{g,n} \| \le \epsilon /4|G|$ and $e'_{g,n}\perp e'_{h,n}$ for $g\neq h$. Thus for a finite set $F$, we can assume that $ \|e'_{g,n}a -ae'_{g,n} \| < \epsilon$ for $a\in F$.  Moreover $\|\alpha_h(e'_{g,n})-e'_{hg,n}) \| < \epsilon$ for such $n$. Finally, if necessary, by changing $r_n$ with a projection we have that both $\| (1-\sum_{g} e'_{g,n})-v^*_nv_n\|$ and $\|r_n-v_nv^*_n\|$  are small enough. Thus we have a partial isometry $w_n$ in $A$ such that $\|v_n- w_n\| < \epsilon$, $1-\sum_g e'_{g,n}=w^*_nw_n$, $w_nw^*_n=r_n \in \overline{xAx}$. So we are done.   
\end{proof}

Let $\sigma:G \curvearrowright C(G)$ be the canonical translation action. We need the following theorem. 

\begin{thm}\cite[Theorem 4.4]{LeeOsaka}\label{T:tracialRokhlinaction2}
Let $G$ be a finite group and $A$ an infinite dimensional simple separable unital finite $C\sp*$-algebra. Then $\alpha:G \curvearrowright A$ has the tracial Rokhlin property if  and only if for every nonzero  element $x$ in $A^{++}_{\infty}$ there exists a $*$-equivariant homomorphism $\phi$ from $(C(G), \sigma)$ to $(A_{\infty}\cap A', \alpha_{\infty})$ such that $1_{A_{\infty}}-\phi(1_{C(G)})$ is Murray-von Neumann equivalent to a projection in a hereditary $C\sp*$-subalgebra generated by $x$ in $A_{\infty}$.  
\end{thm}

\begin{cor}\cite[Corollary 4.6]{LeeOsaka}
Let $G$ be a finite group and $A$ an infinite dimensional simple separable unital finite $C\sp*$-algebra. Assume  $\alpha:G \curvearrowright A$ has the tracial Rokhlin property. Then the map $1_{C(G)}\otimes \id_A : (A, \alpha) \to (C(G)\otimes A, \sigma \otimes \alpha)$ is $G$-tracially sequentially-split. 
\end{cor}
\begin{proof}
$(A_{\infty}\otimes A, \alpha_{\infty}\otimes \alpha)$ can be identified with $(A_{\infty}, \alpha_{\infty})$ by the map sending $\mathbf{a}\otimes x$ to $\mathbf{a}x$. Then we can easily show that $\phi\otimes \id_A$ is the equivariant tracial approximate inverse for every nonzero $x$ in $A^{++}_{\infty}$ where $\phi$ is the equivariant $*$-homomorphism from $C(G), \sigma) \to (A_{\infty}, \alpha_{\infty})$ corresponding to $x$ by Theorem \ref{T:tracialRokhlinaction2}. 
\end{proof}

Then we reprove that $A\rtimes_{\alpha}G$ inherits the interesting approximation properties from $A$ when a finite group action of $G$ has the tracial Rokhlin property  through the notion of tracially sequentially-split map as one of our main results.  We denote by $\phi\rtimes G$ a map from $A\rtimes_{\alpha} G$ to $B\rtimes_{\beta} G$ as a natural extension of an equivariant map $\phi:(A,\alpha) \to (B,\beta)$ where $\alpha:G \curvearrowright A$ and $\beta:G \curvearrowright B$. In the following, we denote $u:G \to U(A\rtimes_{\alpha}G)$ as the implementing unitary representation for the action $\alpha$ so that we write an element of $A\rtimes_{\alpha}G$ as $\sum_{g\in G}a_gu_g$. The embedding of $A$ into $A\rtimes_{\alpha}G$ is $a \mapsto au_e$. 

\begin{lem}\cite[Proposition 1.12]{Phillips:tracial}\label{L:projection}
Let $A$ be an infinite dimensional simple unital $C\sp*$-algebra with Property (SP), and $\alpha:G\curvearrowright A$ be an action of a finite group $G$ on $A$ such that $A\rtimes_{\alpha}G$ is also simple.  Let $B \subset A\rtimes_{\alpha}G$ be a nonzero hereditary $C\sp*$-subalgebra. Then there exists a nonzero projection $p\in A$ which is Murray-von Neumann equivalent to a projection in $B$ in $A\rtimes_{\alpha}G$.  
\end{lem}

\begin{thm}\label{C:GTSS}
Let $G$ be a finite group and $A$  a separable  infinite dimensional simple unital finite $C\sp*$-algebra. Suppose that $\alpha:G \curvearrowright A$ has the tracial Rokhlin property. Then the map $ (1_{C(G)} \otimes id_A) \rtimes  G$ from $A\rtimes_{\alpha} G$ to $(C(G)\otimes A)\rtimes_{\sigma\otimes\alpha} G$ is a tracially sequentially-split $*$-homomorphism. 
\end{thm}
\begin{proof}
We may assume that $A$ has Property (SP). It follows that $A_{\infty}$  has  a weak (SP)-property.  Take a nonzero positive element $z$ in $(A\rtimes_{\alpha}G)^{++}_{\infty} $.  Note that $\alpha:G\curvearrowright A$ is outer so that $A\rtimes_{\alpha}G$ is simple. Then there is a projection $p\in A^{++}_{\infty}$ which is Murray-von Neumann equivalent to a projection in $\overline{z (A\rtimes_{\alpha}G)_{\infty} z}$ by Lemma \ref{L:projection}. Then we take $\phi:C(G) \to A_{\infty}$ such that  $1_{A_{\infty}}-\phi(1_{C(G)})$ is Murray-von Neumann equivalent to a projection $r$ in $pA_{\infty}p$.  Since $r\le p$ and $\overline{z (A\rtimes_{\alpha}G)_{\infty} z}$ is a hereditary $C\sp*$-algebra, $1_{A_{\infty}}-\phi(1_{C(G)})$ is Murray-von Neumann equivalent to a projection in $\overline{z (A\rtimes_{\alpha}G)_{\infty} z}$  by Lemma \ref{L:sub}. \\
Consider the following diagram depending on $z$ 
\[ \xymatrix{ (A,\alpha) \ar[rd]_{1\otimes \id_A} \ar@{-->}[rr] && (A_{\infty}, \alpha_{\infty}) &\\
                          & (C(G)\otimes A, \sigma\otimes \alpha) \ar[ur]_{\phi\otimes \id_A} } \]  
By applying the crossed product functor, we obtain 
\[ \xymatrix{ (A\rtimes_{\alpha}G) \ar[rd]_{(1 \otimes \id_A)\rtimes G} \ar@{-->}[rr]^{\iota\rtimes G} && (A_{\infty}\rtimes_{ \alpha_{\infty}} G) \to (A\rtimes_{\alpha} G)_{\infty} &\\
                          & (C(G)\otimes A)\rtimes_{\sigma\otimes \alpha}G \ar[ur]_{(\phi\otimes \id_A) \rtimes G} } \]  
Here the map  $(A_{\infty}\rtimes_{ \alpha_{\infty}} G) \to (A\rtimes_{\alpha} G)_{\infty}$ is a natural extension of the embedding $A\rtimes_{\alpha}G \to (A\rtimes_{\alpha}G)_{\infty}$. 
 
Note that  $(\widetilde{\phi}\rtimes G) \circ (1_{C(G)}\otimes id_A)\rtimes G= [1_{A_{\infty}}-\phi(1_{C(G)})]u_e \iota_{A\rtimes_{\alpha}G} $ and $[1_{A_{\infty}}-\phi(1_{C(G)})]u_e$ is Murray-von Neumann equivalent to a projection in $\overline{z (A_{\infty} \rtimes_{\alpha_{\infty}} G)z}$. Thus  $(1_{C(G)}\otimes id_A)\rtimes G$ is tracially sequentially split.\\
\end{proof}

Most of the following results are already known due to  Archey \cite{Ar}, Hirshberg and Orovitz \cite{HO},  Kishimoto \cite{Kishi4}, Lin and Osaka \cite{LO:tracial},  Osaka and Teruya \cite{OT1}, Phillips \cite{Phillips:tracial}. We believe that it is meaningful to give a new and unified proof of scattered results using the new conceptual approach. 

\begin{cor}
Let $G$ be a finite group and $A$ be an infinite dimensional separable simple unital finite $C\sp*$-algebra. Suppose that $\alpha:G \curvearrowright A$ has the tracial Rokhlin property. Then if $A$ has the following properties, then so does $A\rtimes_{\alpha} G$.  
\begin{enumerate}
\item  stable rank one, 
\item  real rank zero, 
\item $TA\,\mc{C}$ or $TA \, \mc{S}$,
\item  $\mc{Z}$-absorbing provided that $A$ is nuclear,
\item strict comparison.
\end{enumerate}
\end{cor}

\begin{proof}
Note that $A\rtimes_{\alpha}G$ is simple by \cite[Corollary 1.6]{Phillips:tracial}. 
Since 
\[ \begin{split}
 (C(G)\otimes A) \rtimes_{\sigma\otimes \alpha} G &\simeq (C(G)\rtimes_{\sigma}G)\otimes A\\
&\simeq M_{|G|}(\mathbb{C})\otimes A, 
\end{split}\] 
$(C(G)\otimes A) \rtimes_{\sigma\otimes \alpha} G$ is finite when $A$ has stable rank one. Then $A\rtimes_{\alpha}G$ is finite by Theorem \ref{C:GTSS}. 
 Similarly the remaining conclusions follow from  \ref{T:TSR1}, \ref{T:RR0}, \ref{T:TAC}, \ref{T:TAS}, \ref{T:Z-absorbing} and \ref{T:Strictcomparision}, and Theorem \ref{C:GTSS}.
\end{proof}

Another important notion including the group action on a $C\sp*$-algebra as an example is an inclusion of $C\sp*$-algebras of index-finite type due to Watatani \cite{Watatani:index}.  We briefly recall the definition and related properties used later. 

\begin{defn}[Watatani\cite{Watatani:index}]
Let $P\subset A$ be an inclusion of  unital $C\sp*$-algebras and $E:A \to P$ be an conditional expectation. Then we way that $E$ has a quasi-basis if there exist elements $u_k,v_k$ for $k=1,\dots, n$ such that 
\[ x=\sum_{j=1}^nu_jE(v_jx)=\sum_{j=1}^n E(xu_j)v_j.\]
In this case, we define the Watatani index  of $E$ as 
\[\Index E= \sum_{j=1}^n u_jv_j.\] In other words, we say that $E$ has a finite index if there exists a quasi-basis. 
\end{defn}
It is proved that a quasi-basis can be chosen as $\{(u_1, u_1^*), \dots, (u_n, u_n^*)\}$ so that $\Index E$ is a nonzero positive element in $A$ commuting with $A$.  Thus if $A$ is simple,  it is a nonzero positive scalar. From now on we say that a condition expectation $E$ is of index-finite type if $E$ has a finite index. 

  \begin{defn}[Osaka and Teruya \cite{OT1}]
Let $P\subset A$ be an inclusion of separable unital $C\sp*$-algebras such that a conditional expectation $E:A\to P$ is of index-finite type. We say $E$ has the tracial Rohklin property if for any nonzero element $z\in A^{++}_{\infty}$ there is a Rohklin projection $e\in A_{\infty}\cap A'$ such that 
\begin{enumerate}
\item $(\Index E) E_{\infty}(e)=g$ is a nonzero projection in $P_{\infty}$,
\item $1-g$ is Murray-von Neumann equivalent to a projection in the hereditary subalgebra of $A_{\infty}$ generated by $z$ in $A_{\infty}$, 
\item $A\ni x \to xe \in A_{\infty}$ is injective. 
\end{enumerate} 
\end{defn}

As we notice, the third condition is automatically satisfied when $A$ is simple. As in the case of a finite group action with the tracial Rohklin property, if $P\subset A$ is an inclusion of simple $C\sp*$-algebras and a conditional expectation $E:A \to P$ of index-finite type has the tracial Rohklin property, or shortly $P\subset A$ an inclusion with the tracial Rokhlin property, then either $A$ has Property (SP) or $E$ has the Rokhlin property (see \cite[Lemma 4.3]{OT1}). A typical example arises from a finite group action $\alpha$ of $G$ on a $C\sp*$-algebra $A$; let $A^{\alpha}$ be the fixed point algebra then the conditional expectation 
\begin{equation}\label{E:expectation}
E(a)=\frac{1}{|G|} \sum_{g\in G} \alpha_g(a) 
\end{equation}
  is of index-finite type if the action $\alpha:G \curvearrowright A$ is outer. Moreover,  the following observation was obtained by the second author and T. Teruya in \cite{OT1}. 

\begin{prop}\cite[Proposition 4.6]{OT1}\label{P:actiontoinclusion}
Let $G$ be a finite group and $A$ an infinite dimensional separable finite simple unital $C^*$-algebra $A$, and E as in (\ref{E:expectation}).  Then $\alpha: G \curvearrowright A$ has the tracial Rokhlin property if and only if $E$ has the tracial Rokhlin property.
\end{prop}

 We note that in this case $A^{\alpha}$ is strongly Morita equivalent to $A\rtimes_{\alpha}G$, thus if an approximation property is preserved by the  strong Morita equivalence,  and if the inclusion $A^{\alpha} \subset A$ of index-finite type is tracially sequentially-split, then such an  approximation property can be transferred to $A\rtimes_{\alpha}G$ from $A$ when $\alpha$ has the tracial Rokhlin property.   

To show that the inclusion map $P \hookrightarrow A $ is tracially sequentially-split, we begin to restate \cite[Lemma 4.10]{OT1} as follows. 
\begin{lem}\cite[Lemma 4.19]{LeeOsaka}\label{L:inclusiontechnical}
Let $p,q$ be two projections in $P_{\infty}$ and $e\in A_{\infty}\cap A'$ be a projection such that $(\Index E) E_{\infty}(e)$ is a projection in $P_{\infty}\cap P'$. If $pe=ep$ and $q\lesssim pe$ in $A_{\infty}$, then $q \lesssim p$ in $P_{\infty}$
\end{lem}
The following theorem has already appeared in \cite{LeeOsaka} without a proof. We provide a proof here. 
\begin{thm}\cite[Lemma 4.20]{LeeOsaka}\label{T:tracialinclusion}
Let $P \subset A$ be inclusion of simple $C\sp*$-algebras and $A$ be separable. Suppose $E:A \to P$ has the tracial Rokhlin property.  Then for any nonzero positive element $z \in P^{++}_{\infty}$,  there exists a projection $e$ in the central sequence algebra of $A$ such that $(\Index E)E_{\infty}(e)=g$ is a projection such that $1-g$ is Murray-von Neumann equivalent to a projection in $\overline{zP_{\infty}z}$ in $P_{\infty}$.  
\end{thm}
\begin{proof}
We may assume $P$ has Property (SP).  Consider a nonzero positive element $z$ in $P^{++}_{\infty}$. Following the proof of Lemma \ref{L:SP}  we can construct a nonzero projection $q$ in $\overline{zP_{\infty}z} \cap P_{\infty}^{++}$.  Write $q=[(q_k)_k]$ where each $q_k$ is a projection in $P$. 
For each $q_k$ which is full in $A_{\infty}$, then there exists a full projection $e_k \in A_{\infty}\cap A'$ such that $(\Index E)E_{\infty}(e_k)$ is a projection in $P_{\infty}$. Write $e_k=[(e^k_n)_n]$ where $e_n^k \in A^{+}$ for $n=1,2,\dots$. Then 
\begin{align*}
&\limsup_n \|(\Index E)E(e^k_n)(\Index E)E(e^k_n)-(\Index E)E(e^k_n) \|=0,\\
&\limsup_n \| e^k_nq_k -q_ke_n^k \| =0, \\
&\limsup_n \| e^k_n a -ae^k_n \|=0 \quad \text {for  $a \in A$}.
\end{align*}
Now let $F_k$'s be  increasing finite sets such that $\overline{\cup_{k=1}^{\infty}F_k}=A$.
Then for each $k \in \mathbb{N}$, we can choose a subsequence of increasing $n_k$'s such that 
\begin{align*}
 &\|(\Index E)E(e^k_{n_k})(\Index E)E(e^k_{n_k})-(\Index E)E(e^k_{n_k}) \| \le 1/2^k, \\
& \| e^k_{n_k}q_k -q_ke_{n_k}^k \| \le 1/2^k,\\
&\| e^k_{n_k}a -ae^k_{n_k}\| \le 1/2^k \quad \text{for $a \in F_k$}.
\end{align*}
Then  $e_1=[(e^k_{n_k})_k]$ is a projection of $A_{\infty}\cap A'$ such that $e_1 q=qe_1$ and $(\Index E)E_{\infty}(e_1)$ is a nonzero projection. Note that $e_1q \in A^{++}_{\infty}$. Now we take another Rokhlin projection $e_2\in A_{\infty}\cap A'$ such that $(\Index E)E_{\infty}(e_2)=g_2$ is a projection and $1-g_2 \lesssim qe_1$ in $A_{\infty}$.
By Lemma \ref{L:inclusiontechnical},  $1-g_2 \lesssim q$ in $P_{\infty}$. It follows that $1-g_2$ is Murray- von Neumann equivalent to a projection $\overline{qP_{\infty}q} \subset \overline{zP_{\infty}z}$ in $P_{\infty}$. 
\end{proof}

\begin{cor}\cite[Proposition 4.22]{LeeOsaka}\label{C:TSS}
Let $P\subset A$ be an inclusion of unital $C\sp*$-algebras where $A$ is simple and separable and a conditional expectation $E:A \to P$ of index-finite type has the tracial Rokhlin property. Then the inclusion $P \hookrightarrow A$ is a tracially sequentially-split $*$-homomorphism.  
\end{cor}
\begin{proof}
Let $z$ be a nonzero element $P^{++}_{\infty}$. Then there exists a projection $e \in A_{\infty}\cap A'$ such that $(\Index E)E_{\infty}(e)=g$ is a projection and $1-g$ is Murray- von Neumann equivalent to a projection $r \in \overline{zP_{\infty}z}$ in $P_{\infty}$ by Theorem \ref{T:tracialinclusion}.  Then we set up a map $\beta(a)=(\Index E)E_{\infty}(ae)$.  Note that $\beta(p)=(\Index E)E_{\infty}(pe)=p g$ and $1-\beta(1)=1-g$ is Murray-von Neumann equivalent to a projection in $\overline{zP_{\infty}z}$ in $P_{\infty}$. It remains to show that $\beta$ is a $*$-homomorphism. But it follows from a property that $\beta(x)$ is a unique element  in  $P_{\infty}$ such that $xe=\beta(x)e$ for $x \in A$ (see \cite[Lemma 4.7]{OT1} for more details). 
\end{proof}
\begin{cor}
Let $P \subset A$ be an inclusion of separable unital simple $C\sp*$-algebras and assume that a conditional expectation $E:A \to P$ has the tracial Rohklin property.  If $A$ satisfies one the following properties, then $P$ does too. 
\begin{enumerate}
\item finite or stable rank one,
\item  real rank zero, 
\item $TA\,\mc{C}$ or  $TA \, \mc{S}$,
\item $\mc{Z}$-absorbing provided that $A$ is nuclear,
\item strict comparison.
\end{enumerate}

\end{cor}
\begin{proof}
It follows from Theorems \ref{T:TSR1}, \ref{T:RR0}, \ref{T:TAC}, \ref{T:TAS}, \ref{T:Z-absorbing} and \ref{T:Strictcomparision}, and Corollary \ref{C:TSS}.

\end{proof}
\section{Acknowledgements}
This research was carried out during the first author's stay at KIAS and his visit to Ritsumeikan University. He would like to appreciate both institutions for excellent supports. Both authors are grateful to a colleague for pointing out a flaw in our original definition about ``tracially sequentially-split $*$-homomorphism'' and Chris Phillips for suggesting an outline to fill a gap in the proof of Theorem \ref{T:Strictcomparision}. In addition, we are grateful to T. Teruya for a helpful discussion. 


\end{document}